\documentclass[reqno,11pt]{amsart}
\numberwithin{equation}{section}

\usepackage[utf8]{inputenc}
\usepackage[italian,english]{babel}
\usepackage{amsmath}
\usepackage{amsfonts}
\usepackage{amssymb,esint}
\usepackage{tikz}
\usepackage{bigints}
\usepackage{comment}
\usepackage{geometry}
\usepackage{color}
\usepackage{mathrsfs}
\usepackage{cancel}
\usepackage{mathtools}

\usepackage{calrsfs}

\mathtoolsset{showonlyrefs}

\usepackage[colorlinks, citecolor=citegreen, linkcolor=refred, urlcolor=blue]{hyperref}
\definecolor{citegreen}{rgb}{0,0.4,0}
\definecolor{refred}{rgb}{0.5,0,0}

\theoremstyle{plain}
\newtheorem {theorem}{Theorem}[section]
\newtheorem {lemma}[theorem]{Lemma}
\newtheorem {proposition} [theorem]{Proposition}
\newtheorem {corollary} [theorem]{Corollary}

\newtheorem{definition}[theorem]{Definition}
\theoremstyle{remark}
\newtheorem{remark}[theorem]{Remark}

\DeclarePairedDelimiter\abs{\lvert}{\rvert}

\newcommand{\R}{\mathbb R}
\newcommand{\N}{\mathbb N}

\renewcommand{\H}{\mathbb H}

\renewcommand{\theta}{\vartheta}

\newcommand{\barint}
{\rule[.036in]{.12in}{.009in}\kern-.16in \displaystyle\int}

\usepackage{color}

\newcommand{\numberset}{\mathbb}
\renewcommand{\N}{\numberset{N}}
\renewcommand{\R}{\numberset{R}}

\newcommand{\Heis}{\numberset{H}}

\newcommand{\plap}{\Delta^{\!\Heis^n}_p}

\renewcommand{\phi}{\varphi}
\renewcommand{\epsilon}{\varepsilon}

\title[Strict starshapedness of solutions to the horizontal p-Laplacian]
{Strict starshapedness of solutions to the horizontal p-Laplacian in the Heisenberg group}
\dedicatory{Dedicated to Alberto Farina on the occasion of his 50th birthday}
\author[M.~Fogagnolo]{Mattia Fogagnolo}
\address{M.~Fogagnolo, Centro di Ricerca Matematica Ennio De Giorgi, Scuola Normale Superiore,
Piazza dei Cavalieri 3, 56126 Pisa (PI), Italy}
\email{mattia.fogagnolo@sns.it}

\author[A.~Pinamonti]{Andrea Pinamonti}
\address{A.~Pinamonti, Universit\`a degli Studi di Trento,
via Sommarive 14, 38123 Povo (TN), Italy}
\email{andrea.pinamonti@unitn.it}

\begin{document}

\begin{abstract}
We examine the geometry of the level sets of particular horizontally $p$-harmonic functions in the Heisenberg group. We find sharp, natural geometric conditions ensuring  that the level sets of the $p$-capacitary potential of a bounded annulus are strictly starshaped.
\end{abstract}

\maketitle

\section{Introduction}
The study of the geometric properties of the level sets of solutions to elliptic or parabolic boundary value problems is a classical but still very fertile field of research. Let us focus, without aiming to be complete, on the works most deeply linked with the object of the present paper. Starting in the classical ambient $\R^n$, consider bounded open sets $\Omega_1 \subset \Omega_2$ and let $u$ be a $p$-harmonic function in $\Omega_2 \setminus   \overline{\Omega}_1$ attaining in some sense the value $1$ on $\partial \Omega_1$ and $0$ on $\partial \Omega_2$. It is then quite natural to ask whether some geometric properties such as convexity or starshapedness of $\Omega_2$ and $\Omega_1$ are preserved by (the superlevel sets of) $u$. To the authors' knowledge, a first answer in the much easier case of $p=2$, and in the space $\R^3$ dates back to the 30s, when in \cite{gergen-starshaped} it was showed that the super level sets of $u$ are starshaped if $\Omega_1$ and $\Omega_2$ are. Later, in \cite{gabriel-convexity1, gabriel-convexity2, gabriel-convexity3} it was substantially shown that the same phenomenon occurs for the convexity issue. The nonlinear case was arguably first considered in \cite{pfaltzgraff1967}, where it was shown that the $p$-capacitary potentials in starshaped rings are starshaped. It is important to point out that the proof provided in such paper relies on a suitable symmetrization technique, that despite being powerful enough to treat very general equations \cite{salani-starshaped}, does not seem to provide the \emph{strict starshapedness} (a notion to be described in a while) of $u$ from that of $\Omega_1$ and $\Omega_2$. We close this historical excursus mentioning the fundamental \cite{lewis}, where, in addition to the challenging extension of the aforementioned convexity results to the nonlinear setting, the author fine-tunes the maximum principle techniques we are adopting in the present work.

Before venturing in a description of our main result, let us observe that the kind of issue we just briefly discussed has recently gained attention also in the context of sub-Riemannian geometries, that is actually the setting for this article. Indeed, in \cite{danielligarofalo-starshaped} it was shown that in the linear situation $p = 2$ starshapedness  of $\Omega_1$ and $\Omega_2$ is not only preserved by $u$, but actually improved to strict starshapedness. The nonlinear generalization appeared in \cite{dragoni-starshaped}. However, as in the standard Euclidean situation, if $p \neq 2$ it is not clear whether strict starshapedness is preserved.  A $C^1$-domain containing the origin is said to be starshaped (with respect to the origin) in the Heisenberg geometry if its outer unit normal $\nu$ satisfies $\langle \nu , Z \rangle \geq 0$, while it is strictly starshaped if such inequality is strict. 
We are denoting by $Z$ the dilation-generating vector field. 
To fix the ideas, we may think of $Z$ as the natural replacement  for the Euclidean position vector in the classical notion of starshapedness. 


The aim of the present paper is to explore the issue of strict starshapedness for the superlevel sets of $p$-capacitary potentials in the Heisenberg group $\Heis^n$. Leaving the definitions and a brief introduction  to the Heisenberg groups to the next section, let $\Omega_1 \Subset \Omega_2 \Subset \Heis^n$ be $C^1$ domains containing the origin, and consider, for $p > 1$, the solution $u$ to
\begin{equation}
\label{pb-bounded}
\begin{cases}
\,\,\Delta^{\!\Heis^n}_p{u}= 0 & \mbox{in} \,\, \Omega_2\setminus\overline\Omega_1 
\\
\,\,\,\,\,\,\,\,\,\,\,\,\,u=1 & \mbox{on}\,\, \overline\Omega_1 
\\
\,\,\,\,\,\,\,\,\,\,\,\,\,u = 0 &\mbox{on} \,\, \partial\Omega_2,
\end{cases}
\end{equation}
where by $\Delta^{\!\Heis^n}_p$ we indicate the \emph{horizontal} $p$-Laplacian, that is the natural analogue in $\Heis^n$ of the classical $p$-Laplacian in $\R^n$. Our main result substantially establishes that the \emph{strict} starshapedness of $\Omega_1$ and $\Omega_2$ is preserved by $u$ . 
\begin{theorem}
\label{th-starshaped-bounded}
Let $u$ be a $C^1$-weak solution  to \eqref{pb-bounded} with $p > 1$ for $\Omega_1 \subset \Heis^n$ and $\Omega_2 \subset \Heis^n$ bounded sets with $C^1$ boundaries that are \emph{strictly starshaped} with respect to  the origin $O \in \Heis^n$ and such that $\Omega_1 \Subset \Omega_2$. Assume also that $\Omega_1$ satisfies an \emph{uniform exterior gauge ball condition} and $\Omega_2$ satisfies an \emph{uniform interior gauge ball condition}.
Then, $\{u \geq t\} \cup \overline{\Omega}_1$ is a bounded set with $C^1$-boundary that is strictly starshaped with respect to $O$ for any $t \in [0, 1]$. Moreover, $\abs{\nabla u} \neq 0$ in ${\Omega_2} \setminus \overline{\Omega}_1$. 
\end{theorem}
We are actually going to fully prove Theorem \ref{th-starshaped-bounded} for $p \neq Q$, with $Q = 2n + 2$ since, with the techniques adopted, the modifications needed to cover the case $p = Q$ are straightforward, and illustrated in Remark \ref{p=Q} below.

The above result will follow from a somewhat more general principle, Theorem \ref{thm-prestarshaped} below, asserting that, if $\Omega_1$ and $\Omega_2$ satisfy the above conditions, then at any point where the classical gradient of the solution $u$ to \eqref{pb-bounded} exists we have $\langle \nabla u, Z \rangle $ uniformly  bounded away from zero. Theorem \ref{th-starshaped-bounded} becomes then an immediate corollary of such statement. To the authors' knowledge, it is not clear whether weak solutions to \eqref{pb-bounded} actually do enjoy classical $C^1$-regularity for a general $p> 1$. On the other hand, it has been established in \cite{domokos-manfredi} for $p$ belonging to a neighbourhood of $2$. It has been moreover recently discovered in \cite{shirsho-zhong} that at least the \emph{horizontal gradient} is H\"older continuous for any $p > 1$. 

\smallskip

It is worth pointing out that our result does not allow to conclude that the \emph{horizontal} gradient of $u$ does not vanish. This is actually sharp, since in the explicit, symmetric situation where $\Omega_1$ and $\Omega_2$ are two concentric \emph{gauge balls}, i.e. defined with respect to the well known Koranyi norm of $\Heis^n$, the level sets of $u$ remain gauge balls, that in particular display a characteristic point where the horizontal gradient vanishes, while the vertical derivative does not. This is also the reason why we cannot, with the available technology, infer higher regularity of $u$ and its level sets from the conclusion of Theorem \ref{th-starshaped-bounded}. Indeed, it is known from the arguments in \cite{ricciotti-book} that the nonvanishing of the horizontal gradient implies the smoothness of $u$, but it is not known whether the nonvanishing of the vertical derivative alone suffices to this aim.

All in all, the issue of regularity for horizontally $p$-harmonic functions in Heisenberg groups has been and still is a fervid field of research, and we think our work could serve as an additional motivation to carry on with that. In addition to the aforementioned contributions about this topic, we cite the papers \cite{mingione-zhong, manfredi-mingione, capogna-regularity, ricciotti-regularity}. Let us recall briefly that the optimal $C^{1, \alpha}$-regularity of standard $p$-harmonic functions is well known, and established independently in \cite{dibenedetto-p-regularity, tolksdorf-regularity}.

Let us now pass to discuss the last main assumption involved in Theorem \ref{th-starshaped-bounded},  
namely the tangent gauge ball condition we ask $\Omega_1$ and $\Omega_2$ to be subject to.  Such property constitutes clearly a natural analogue of the round tangent ball conditions in the Riemannian geometry, that is indeed heavily used to deal with barrier arguments, as those performed here. On the other hand, differently from classical situations, a tangent gauge ball to the boundary of a set $\Omega \subset \Heis^n$ is not ensured \emph{no matter the regularity of} $\partial \Omega$. This is explicitly shown in \cite{martino-tralli}.

\bigskip

The proof of Theorem \ref{th-starshaped-bounded} is inspired by the barrier argument used in the proof \cite[Lemma 2]{lewis}. Indeed, such result can be substantially rephrased as the standard Euclidean version of Theorem \ref{th-starshaped-bounded}. On the other hand, we emphasize that Lewis' paper deals with convexity assumptions, and thus this is, to our knowledge, the first place where the strict starshapedness is observed to be preserved even in the classical context.
Let us point out that this type of argument has been also finely reworked in the Appendix of \cite{bianchini-ciraolo} to deal with the anisotropic $p$-Laplacian.
Closing these brief comments on the proof, we observe that these ideas, at least from an heuristic point of view, seem to be exportable to more general Carnot groups. On the other hand, as often occurs, technical challenges could arise when dealing with such a generalization. 
We could come back on this topic in future works.

\medskip

Let us comment on possible geometric applications and perspectives of results such as Theorem \ref{th-starshaped-bounded}. Problem \eqref{pb-bounded} for the standard $p$-Laplacian, with $\Omega_2$ dilated away at infinity, has been recently utilized in \cite{Fog_Maz_Pin} and \cite{Ago_Fog_Maz_2} as a substitute for the Inverse Mean Curvature Flow \cite{Hui_Ilm, moser-jems} to infer geometric and analytical inequalities for hypersurfaces in $\R^n$. In the earlier \cite{Fog_Maz_Pin} the nonvanishing of $\abs{\nabla u}$ was assumed in order to establish suitable monotonicity formulas. In particular, by the aforementioned \cite[Lemma 2]{lewis} , the results of \cite{Fog_Maz_Pin} hold true for strictly starshaped domains. Theorem \ref{th-starshaped-bounded} can thus be read as a first indication about the viability of these techniques in the sub-Riemannian setting. More generally, the preserving of starshapedness along suitable evolutions of hypersurfaces is a highly desirable and thoroughly studied property in Geometric Analysis, let us cite for the mere sake of example \cite{gerhardt, urbas, Brendle, pipoli-complex}.

\bigskip

The present paper is structured as follows. In Section \ref{heisenberg-sec} we recall and discuss the preparatory material we are going to need for the proof of Theorem \ref{th-starshaped-bounded}. More precisely, we review definitions and basic properties of the Heisenberg group, we discuss hypersurfaces that fulfil starshapedness and tangent gauge balls conditions, and recall some fundamental facts about horizontally $p$-harmonic functions.
In Section \ref{proof-sec} we work out the proof of Theorem \ref{th-starshaped-bounded}, that as already mentioned will follow from the slightly more general Theorem \ref{thm-prestarshaped}.

\section{The Heisenberg group and horizontally $p$-harmonic functions}
\label{heisenberg-sec}
We summarize below some properties of the Heisenberg group that we will need throughout the paper. We follow here the presentation given in \cite{monti-dispensa} and we address the interested reader to \cite{BLU} for a complete overview.
\subsection{The Heisenberg group}
Let $n \geq 1$. We denote by $\Heis^n$ the Lie group $(\R^{2n + 1}, \cdot)$, where the group product between $z = (x_1, \dots, x_n, y_1, \dots, y_n, t)$ and $\tilde{z} =  (\tilde{x}_1, \dots, \tilde{x}_n, \tilde{y}_1, \dots, \tilde{y}_n, \tilde{t})$ is defined by
\begin{equation}
\label{product}
z \cdot \tilde{z} = \left(x_1 + \tilde{x}_1, \dots, x_n + \tilde{x}_n, y_1 + \tilde{y}_1, \dots, y_n + \tilde{y}_n, t + \tilde{t} - 2 \left(\sum_{i =1}^n x_i \tilde{y}_i + \tilde{x}_i y_i\right) \right).
\end{equation}
We denote by $z^{-1}$ the inverse of $z \in \Heis^n$ with respect to the group law defined above.
The Lie algebra $\mathfrak{g}$ of left invariant vector fields of $\Heis^n$ is spanned by the vector fields
\begin{equation}
\label{lie}
X_i = \frac{\partial}{\partial x_i} + 2y_i \frac{\partial}{\partial t}, \quad Y_j = \frac{\partial}{\partial y_j} - 2x_j\frac{\partial}{\partial t}, \quad T = \frac{\partial}{\partial t} \qquad i, j = 1, \dots, n. 
\end{equation}
It is easy to see that the Lie algebra of $\Heis^n$ is stratified of step $2$, i.e. denoting by 
\begin{align}
    V_1 = \mathrm{span}\{X_i, Y_i,\ i=1,\ldots, n\},\qquad V_2 = \mathrm{span}\{T\}
\end{align}
it holds
\begin{align}
    \mathfrak{g}=V_1\oplus V_2
\end{align}
and $[V_1,V_1]=V_2$, where $[V_1,V_1]=\mathrm{span}\{[v,w],v,w\in V_1\}$.
As usual we refer to $V_1$ as the $\emph{horizontal layer}$ and to $X_i$ and $Y_j$ as \emph{horizontal vector fields}. 

Given an open subset $U \subseteq \Heis^n$, and a function $f \in C^1(U)$, we denote by $\nabla_{\Heis^n} f$ the horizontal gradient of $f$ defined as
\begin{equation}
\label{horizontal-gradient-true}
\nabla_{\mathbb{H}} f= \left(X_1 f, \dots, X_n f, Y_1 f, \dots, Y_n f\right)
\end{equation}
while we denote with $\nabla$ the classical gradient in $\R^{2n + 1}$. 

Given any two vector fields $Z$ and $W$ on $\Heis^n$ we are going to consider its classical scalar product in $\R^{2n + 1}$ and to denote it by $\langle Z, W \rangle$.
Similarly, we indicate the  $\R^{2n + 1}$-norm of $Z$ simply by $\abs{Z}$.
 
Various distances are usually considered in relation with Heisenberg groups, but we limit ourselves to the one induced by the Koranyi homogeneous norm. Given $z, w \in \Heis^n$ with coordinates as above, we define the Koranyi homogeneous norm as follows
\begin{equation}
\label{gauge1}
\rho(z)  =  \rho(x_1,\ldots, x_n, y_1,\ldots, y_n, t) = \left[\left(\sum_{i= 1}^n x_i^2 + y_i^2\right)^2 + t^2\right]^{\frac{1}{4}} 
\end{equation} 
and then the Koranyi, or \emph{gauge}, distance between $z$ and $w$ is simply $\rho(w^{-1} \cdot z)$.
Accordingly, we denote by $B_\rho (z, R)$, for $R > 0$,  the open ball with respect to such metric, that is
\begin{equation}
\label{gauge-ball}
B_\rho (z, R) = \{w \in \Heis^n \, \vert \, \rho (w^{-1} z) < R\}.
\end{equation}
We refer to $B_\rho(z, R)$ as the \emph{gauge ball} of center $z$ and radius $R$.\\

The proof of our main result, i.e. Theorem \ref{th-starshaped-bounded}, requires suitable dilations of subsets and functions. To this end, let us recall the dilation $\delta_\lambda : \Heis^n \to \Heis^n$ of parameter $\lambda\in\mathbb{R}$ defined by
\begin{equation}
\label{dilations}
\delta_\lambda (x_1,, \dots, x_n, y_1, \dots, y_n, t) = (e^\lambda x_1, \dots, e^\lambda x_n, e^\lambda y_1, \dots, e^\lambda, y_n, e^{2\lambda}t).
\end{equation}
With such definition at hand, we introduce the following notations for dilated subsets and dilated functions.
For a subset $A \subset \Heis^n$ and $\lambda \in \R$, we define its dilation as
\begin{equation}
\label{dilation-subset}
A^\lambda \, = \, \delta_\lambda (A) \, = \, \{\delta_\lambda (z) \quad \vert  \, \, z \in A\}.
\end{equation}
Moreover, for a function $f :A \to \R$, we define its dilated $f_\lambda : A^{-\lambda} \to \R$ by 
\begin{equation}
\label{dilation-function}
f_\lambda (z) = f(\delta_\lambda (z)).
\end{equation}


\subsection{Starshaped sets and gauge ball properties}
The following is the definition we adopt for starshaped sets in the Heisenberg group, well posed for sets with $C^1$-boundary. We address the reader to \cite{dragoni-starshaped1} and \cite{dragoni-starshaped} for more extensive discussions on this geometric property in the more general context of Carnot groups and for some equivalent definitions.
\begin{definition}[Starshaped and strictly starshaped sets in Heisenberg groups]
\label{starshaped}
An open bounded set $\Omega \subset \Heis^n$ with $C^1$ boundary is \emph{starshaped} with respect to the origin if it contains the origin and
\begin{equation}
\label{starshaped-condition}
\left\langle \nu, Z\right\rangle (z) \geq 0
\end{equation}
for any $z \in \partial \Omega$, where $\nu$ is the Euclidean exterior unit normal to $\partial \Omega$, and 
\[
Z = (x_1, \dots, x_n, y_1, \dots, y_n, 2t)
\] 
is the dilation-generating vector field. The set $\Omega$ is called \emph{strictly starshaped} if inequality \eqref{starshaped-condition} holds with strict sign at any $z \in \partial \Omega$.
\end{definition}
It is worth observing the well-known relation between the vector field $Z$ appearing above and the geometry of $\Heis^n$. As the name of dilation-generating vector field suggests, the flow of $Z$ is given by the dilation map $\Heis^n \times \R \to \Heis^n$ given by $(z, \lambda) \to \delta_\lambda(z)$. In other words, we have
\begin{equation}
\label{flow-z}
\frac{d}{d \lambda} \delta_\lambda (z) = Z(\delta_\lambda (z)) , \qquad \delta_0(z)=z
\end{equation}
for any $z \in \Heis^n, \lambda \in \R$.

We now recall the definition of boundaries satisfying an interior or exterior gauge-ball condition. 
\begin{definition}[Interior and exterior gauge-ball property]
\label{gauge-ball-def}
Let $\Omega \subset \Heis^n$ be an open bounded set. Then, we say that $\Omega$ satisfies the \emph{exterior gauge ball property} at a point $z \in \partial\Omega$ if there exist $z_1 \in \Heis^n \setminus \overline{\Omega}$  and $R_1 > 0$ such that $B_\rho(z_1, R) \subset \Heis^n \setminus \overline{\Omega}$ and $\overline{B}_\rho (z_1, R) \cap \partial \Omega = \{z\}$. We say that $\Omega$ satisfies a \emph{uniform} exterior gauge ball property if $R$ is uniform for $z \in \partial \Omega$. 

Similarly we say that $\Omega$ satisfies the \emph{interior ball condition} at a point $z \in \partial \Omega$ if there exist  $z_2 \in \Omega$  and $R_2 > 0$ such that $B_\rho(z_2,R_2) \subset \Omega$ and $\overline{B}_\rho (z_2, R_2) \cap \partial \Omega = \{z\}$. Again, we say that $\Omega$ satisfies a \emph{uniform} interior gauge ball property if $R$ is uniform for $z \in \partial \Omega$.  
\end{definition}
It is worth pointing out that, strikingly differently from the Euclidean case, one can easily find in $\Heis^n$ sets with smooth boundary not satisfying exterior or interior gauge ball conditions. This is ultimately due to the lack of strict, uniform Euclidean convexity of gauge balls, that is, their boundaries display points where some principal curvature (with respect to the flat metric of $\R^{2n + 1}$) vanishes. An explicit example of domains with smooth boundary not satisfying gauge ball conditions is shown in \cite{martino-tralli}. Indeed,  in such paper it is shown, precisely in \cite[Proposition 2.4]{martino-tralli}, that an interior gauge ball condition at a point $z \in \partial \Omega$, for some  open set $\Omega$, suffices to prove a Hopf boundary point lemma for harmonic functions at $z$. On the other hand, the authors provide in \cite[Counterexample 2.3]{martino-tralli} a set with paraboloidal boundary in $\Heis^1$ with the Hopf property failing on the vertex, that in particular does not admit a gauge ball touching from the inside. 

It is also important to remark that while it is very easy to find sets with characteristic boundary points satisfying gauge ball condition (gauge balls themselves provide such examples), non-characteristic points of $C^{1, 1}$ boundaries enjoy exterior and interior touching ball condition. This is shown in \cite[Theorem 8.4]{garofalo-phuc}. In particular, the conditions of Definition \ref{gauge-ball-def} are strictly weaker than being non-characteristic. Finally, we address the interested reader to \cite{Capogna,Jerison1,Jerison2,Lanc} and references therein for some discussions on the importance of the gauge-ball property from the regularity standpoint.
 
In the proof of Theorem \ref{th-starshaped-bounded}, we are using the following natural property of (strictly) starshaped sets. It substantially consists in a refinement, holding true under the additional assumptions of exterior or interior gauge ball conditions, of similar properties described in the more general context of Carnot groups in \cite[Proposition 4.2]{dragoni-starshaped}.
\begin{proposition}
\label{entrata-palla}
Let $\Omega \subset \Heis^n$ be a strictly starshaped open bounded set with $C^1$-boundary. Assume that $\Omega$  satisfies the exterior gauge ball condition at $z \in \partial \Omega$, and let $B_\rho (z_1, R)$ the exterior tangent gauge ball at $z$. Then, there exists $\overline{\lambda} > 0$ such that for any $0 < \lambda \leq \overline{\lambda}$ we have $\delta_\lambda (z) \in B_\rho(z_1, R)$.

Analogously if $\Omega$ admits an interior tangent gauge ball $B(z_2, R)$ at $z \in \partial \Omega$, then there exists $\overline{\lambda} > 0$ such that any $-\overline{\lambda} < - \lambda < 0$ we have $\delta_{-\lambda} (z) \in B_\rho(z_2, R)$.
In both cases, the constant $\overline{\lambda}$ depends continuously on $R$, on the center of the ball and on the point of tangency.
\end{proposition}
\begin{proof}
Let $B_\rho (z_1, R)$ be a gauge ball of center $z_1$ and radius $R$, and let $z \in\partial B_\rho (z_1, R)$. If $\left\langle \nu, Z\right\rangle > 0$, where $\nu$ is the interior normal to $B_\rho (z, R)$, as in the first case of the statement we are proving, then this means that $Z$ points towards the interior of the gauge ball, and thus, since $\delta_\lambda$ satisfies \eqref{flow-z}, we infer the existence of $\overline{\lambda} > 0$ such that $\delta_\lambda(z) \in B_\rho(z_1, R)$ for any $0 < \lambda < \overline{\lambda}$. From the continuity of the map $(z, \lambda) \to \delta_\lambda(z)$, following from \eqref{flow-z}, we also deduce that $\overline{\lambda}$ changes continuously with respect to the data. The statement about interior tangent balls is shown the very same way.
\end{proof}
\subsection{Preliminaries on $p$-harmonic functions in the Heisenberg group}
Let $1<p<\infty$, we denote the horizontal $p$-Laplacian with $\plap$. It acts on a $C^2$ function $f$ of $\Heis^n$ as
\begin{equation}
\label{smooth-plap}
\plap f = \sum_{i = 1}^n X_i \left(\abs{\nabla_{\H^n} f}^{p-2} X_i f\right) + Y_i \left(\abs{\nabla_{\H^n} f}^{p-2} Y_i f\right).
\end{equation}
Consequently, we say that a $C^2$-function $f:U\subset\H^n\to \mathbb{R}$ is \emph{horizontally $p$-harmonic} in an open set $U$ if $\plap f = 0$ in $U$.

\subsubsection*{Explicit solutions}
We immediately exhibit explicit horizontally $p$-harmonic functions, that will serve us both as model solutions to \eqref{pb-bounded} and to construct the barriers functions employed in the proof of Theorem \ref{th-starshaped-bounded}. 
We have that, if $p \neq Q$, for any $w \in \Heis^n$, the function 
\begin{equation}
\label{green}
v_w (z) = \rho^{-\frac{Q - p}{p-1}} (w^{-1} z)
\end{equation}
is horizontally $p$-harmonic for any $z \in \Heis^n \setminus \{w\}$. This was established in \cite[Theorem 2.1]{garofalo-explicit-green}, where the authors showed that, up to a normalizing constant, the function $G : \Heis^n \times \Heis^n \setminus \mathrm{Diag} \, {\Heis^n}$ defined by $G (z, w) = v_w(z)$ constitutes the fundamental solution for the horizontal $p$-Laplacian with singularity at $\mathrm{Diag} \, {\Heis^n}=\{(z,z)\in \Heis^n\times \Heis^n\}$.
In particular, as pointed out in the same paper, the (unique) solution to problem \eqref{pb-bounded} in the model situation where $\Omega_1 = B_\rho (O, r)$ and $\Omega_2 = B_\rho (O, R)$ for some $R > r$, is given by 
\begin{equation}
\label{model}
u_{r, R}(z) = \frac{\rho(z)^{-\frac{Q-p}{p-1}} - R^{-\frac{Q-p}{p-1}}}{r^{-\frac{Q-p}{p-1}} -R^{-\frac{Q-p}{p-1}}}
\end{equation}
 for any $p \neq Q$.
\begin{remark}
\label{p=Q}
In the case $p = Q$, the analogue of \eqref{green}, again according to \cite{garofalo-explicit-green}, is given by
\begin{equation}
\label{green-Q}
v_w(z) = \log \rho(w^{-1} \cdot z).
\end{equation}
In particular, the proof of the analogue of Theorem \ref{pb-bounded} in the case $p = Q$ is obtained simply by modelling the barrier functions employed in our proof on \eqref{green-Q} rather than on \eqref{model}. We address the interested reader to \cite{lewis-reg} for more details on how to handle this situation in the case of $\R^n$ with the standard notion of $p$-Laplacian.  
\end{remark}
Let us recall now some functions spaces suited to define the weak solutions to 
\eqref{smooth-plap}. We follow \cite{garofalo-harnack} and \cite{danielli-harnack}. We let, for an open subset $U \subseteq \Heis^n$, for $p \geq 1$, the \emph{horizontal $(1, p)$-Sobolev space} $HW^{1, p}(U)$ be defined as the metric completion of $C^1(U)$ in the norm
\begin{equation}
\label{hw1p-norm}
\vert\vert f \vert\vert_{HW^{1, p}(U)} = \int_U \abs{f}^p + \abs{\nabla_{\Heis^n} f}^p \, dz .
\end{equation}
Analogously, we define the space $HW^{1, p}_0 (U)$ as the metric completion of $C^1_0 (U)$ under the same norm.

We say that $f \in HW^{1, p}(U)$ is \emph{horizontally weakly $p$-harmonic} if
\begin{equation}
\label{weak}
\sum_{i = 1}^n \int_U  \left(\abs{\nabla_{\H^n} f}^{p-2} X_i f\right)X_i \phi +  \left(\abs{\nabla_{\H^n} f}^{p-2} Y_i f\right) Y_i \phi \,  \, d z = 0
\end{equation}
for any $\phi \in C^1_0 (U)$. 
From now on, we will frequently indicate horizontally weakly $p$-harmonic functions simply as $p$-harmonic, since no confusion can occur.

\smallskip

  By arguing exactly as in the Euclidean case, one recovers the fundamental Comparison Principle even for horizontally weakly $p$-harmonic functions in the Heisenberg group. We address the reader to \cite[Lemma 2.6]{danielli-harnack} for a statement in the more general context of quasilinear equations in Carnot groups. It actually holds also comparing subsolutions to supersolutions of the $p$-Laplacian, but being here concerned only with $p$-harmonic functions, we state it in the simplified version for solutions.
\begin{proposition}[Comparison Principle for $p$-harmonic functions]
\label{comparison}
Let $U \subset \Heis^n$ be an open set, and let $u, v \in HW^{1, p}(U)$ be $p$-harmonic functions. Then, if $\mathrm{min}(u - v, 0) \in HW^{1, p}_0 (U)$, then $u \geq v$ on the whole $U$.
\end{proposition}
The above result roughly asserts that if two $p$-harmonic functions $u$ and $v$ satisfy $ u \geq v$ on the boundary of $U$ then the same inequality holds true in the interior on $U$. Actually, this is exactly what happens when boundary data are attained with some regularity.

\smallskip

Finally, let us recall that as an immediate consequence of the Harnack inequality for $p$-harmonic functions in Heisenberg groups \cite[Theorem 3.1]{garofalo-harnack} we get the following special form of a Strong Maximum/Minimum Principle, highlighted also in \cite[Theorem 2.5]{danielli-harnack}.
\begin{proposition}[Strong Maximum Principle for $p$-harmonic functions]
\label{strong}
Let $U \subset \Heis^n$ be an open bounded subset, and let $u \in HW^{1, p}(U)$ be $p$-harmonic. Then, $u$ cannot achieve neither its maximum nor its minimum in $U$.
\end{proposition}

\subsubsection*{Existence and uniqueness for \eqref{pb-bounded}}
In the following statement we resume an existence-uniqueness theorem for problem \eqref{pb-bounded}, recalling also a suitable definition of weak solutions. It is well known that such solution exists, and can be proved exactly as in the Euclidean case, considered in full details in \cite[Appendix I]{heinonen-book}, see in particular Corollary 17.3 there, and compare also with \cite[Section 3]{danielli-harnack}. The uniqueness immediately follows from the Comparison Principle recalled in Proposition \ref{comparison}.

\begin{theorem}[Existence and uniqueness of weak solutions to problem \eqref{pb-bounded}]
\label{existence}
Let $\Omega_1$ and $\Omega_2$ and $\Omega_2 \subset \Heis^n$ be open bounded subsets of $\Heis^n$ satisfying $\Omega_1 \Subset \Omega_2$. Then, there exists an unique weak solution $u$ to \eqref{pb-bounded}, that is $u \in HW^{1, p}(\Omega_2 \setminus \overline{\Omega}_1)$ is horizontally weakly $p$-harmonic and,  letting $\theta \in C^{\infty}_0(\Omega_2)$ satisfy $\theta \equiv 1$ on $\overline{\Omega}_1$, we have $u - \theta \in HW^{1, p}_0(\Omega_2 \setminus \overline{\Omega}_1)$.
\end{theorem}
It is important to point out that, again as a straightforward application of the Comparison Principle, if $\tilde u$ is another such function satisfying the properties in the statement of Theorem \ref{existence} relatively to another boundary datum $\tilde{\theta}$ fulfilling the same assumptions asked for $\theta$, then $u$ coincides with $\tilde u$ on $\Omega_2 \setminus \overline{\Omega}_1$. This is observed with some more details for example in \cite[p. 115]{heinonen-book}.

For what it concerns the continuous attainment of the boundary datum, we point out that in \cite[Theorem 3.9]{danielli-harnack} continuity up to the boundary for Dirichlet problems involving the horizontal $p$-Laplacian is proved for domains with boundary with a so-called \emph{corkscrew} on any point of the boundary. 

\medskip

Let us finally observe that as a consequence of Propositions \ref{comparison} and \ref{strong} we have $0 < u < 1$ on $\Omega_2 \setminus \overline{\Omega}_1$. Indeed, first observe that, since $u = \theta + f$ for some $f \in HW_0^{1, p}$ we can find by approximating $f$  a sequence $\{u_k\}_{k \in \N}$ of functions in $C^1 (\Omega_2 \setminus \overline{\Omega}_1)$ approximating $u$ in $HW^{1, p}$-norm, and satisfying $u_k - \theta \in C^1_c (\Omega_2 \setminus \overline{\Omega}_1)$. In particular, for any $k \in \N$, $u_k$ satisfies $\mathrm{min}{(u_k, 0)} \in HW^{1, p}_0 (\Omega_2 \setminus \overline{\Omega}_1)$, and thus, passing to the limit as $k \to \infty$, we infer that the same holds for $u$. Thus, being $u$ $p$-harmonic, we get from the Comparison Principle recalled in Proposition \ref{comparison} that $u \geq 0$ on the annulus. Arguing in the same way for the $p$-harmonic function $1 - u$, we also find that $u \leq 1$ on $\Omega_2 \setminus \overline{\Omega}_1$. However by the Strong Maximum Principle of Proposition \ref{strong}, the inequalities $0 \leq u \leq 1$ must be strict, as claimed.
We record what has just been said in the following corollary.
\begin{corollary}
\label{u-strict}
Let $u$ be the solution to \eqref{pb-bounded}, in the sense of Theorem \ref{existence}. Then, we have $0 < u <1$ in ${\Omega_2} \setminus \overline{\Omega}_1$.
\end{corollary}

\section{Proof of Theorem \ref{th-starshaped-bounded}}
\label{proof-sec}
It is quite straightforward, but fundamental for our arguments, to observe that if a function $f$ is $p$-harmonic, then so does the dilated $f_\lambda$ defined as $f_\lambda (z) = f (\delta_\lambda (z))$. 
\begin{lemma}[Dilation-invariance of $p$-harmonicity]
\label{dilation-invariance}
Let $U \subseteq \Heis^n$, and $f \in HW^{1, p} (U)$ be a $p$-harmonic function. Then, the function $f_\lambda (x)$ belongs to $HW^{1, p}(\delta_{-\lambda}(U))$ and it is $p$-harmonic.
\end{lemma}
\begin{proof}
It is obvious from the definition of $\delta_{\lambda}(U)$ that $f_\lambda$ is well defined on such set. In order to prove the other assertions, the main computation is the following. We have, for $j = 1, \dots, n$,
\begin{equation}
\label{xjlambda}
X_j(f_\lambda (z)) = e^\lambda\left[\frac{\partial f}{\partial x_j} (\delta_\lambda (z)) + 2 e^\lambda y \frac{\partial f}{\partial t} (\delta_\lambda(z))\right] = e^\lambda \left(X_j f\right)_\lambda (z),
\end{equation} 
and analogously
\begin{equation}
\label{yjlambda}
Y_j(f_\lambda (z)) = e^\lambda\left[\frac{\partial f}{\partial y_j} (\delta_\lambda (z)) - 2e^\lambda x  \frac{\partial f}{\partial t} (\delta_\lambda(z))\right] = e^\lambda \left(Y_j f\right)_\lambda (z).
\end{equation}
The inclusion of $f_\lambda$ in $HW^{1, p}(\delta_\lambda (U))$ is a direct consequence of \eqref{xjlambda} and \eqref{yjlambda}, while  the $p$-harmonicity is shown as follows. We have, again as a consequence of the above relations
\begin{equation}
\label{p-harm-lambda}
\begin{split}
\sum_{i = 1}^n \int_{\delta_{-\lambda} (U)} &\left(\abs{\nabla_{\H^n} f_\lambda}^{p-2} X_i  f_\lambda \right)X_i \phi +  \left(\abs{\nabla_{\H^n} f_\lambda}^{p-2} Y_i f_\lambda\right) Y_i \phi \, \, d z =\\
& = e^{\lambda(p-1)} \sum_{i = 1}^n \int_{\delta_{-\lambda} (U)} \left[\abs{\nabla_{\H^n} f}_\lambda^{p-2} (X_i  f)_\lambda \right]X_i \phi +  \left[\abs{\nabla_{\H^n} f}_\lambda^{p-2} (Y_i f)_\lambda\right] Y_i \phi \, \, d z \\
& = e^{\lambda(p-1) - Q} \sum_{i = 1}^n \int_{U} \left[\abs{\nabla_{\H^n} f}^{p-2} (X_i  f) \right](X_i \phi)_\lambda +  \left[\abs{\nabla_{\H^n} f}\lambda^{p-2} (Y_i f)\right] (Y_i \phi)_\lambda \, \, d z \\
& = e^{\lambda(p-2) - Q} \sum_{i = 1}^n \int_{U} \left[\abs{\nabla_{\H^n} f}^{p-2} (X_i  f) \right]X_i \phi_\lambda +  \left[\abs{\nabla_{\H^n} f}\lambda^{p-2} (Y_i f)\right] Y_i \phi_\lambda \, \, d z \\
& = 0
\end{split}
\end{equation}
for any $\phi \in C^1_c(\delta_\lambda (U))$. The last step follows from the $p$-harmonicity of $f$ in $U$, since $\phi_\lambda$ clearly belongs to $C^1_c(U)$. 
\end{proof}

We are finally in position to prove the statement in turn implying Theorem \ref{th-starshaped-bounded}.
\begin{theorem}
\label{thm-prestarshaped}
Let $u$ be a weak solution  to \eqref{pb-bounded} with $p > 1$ for $\Omega_1 \subset \Heis^n$ and $\Omega_2 \subset \Heis^n$ bounded sets with $C^1$ boundaries that are \emph{strictly starshaped} with respect to  the origin $O \in \Heis^n$ and such that $\Omega_1 \Subset \Omega_2$. Assume also that $\Omega_1$ satisfies an \emph{uniform exterior gauge ball condition} and $\Omega_2$ satisfies an \emph{uniform interior gauge ball condition}.
Then, there exists a positive constant $M$ such that 
\begin{equation}
\label{prestarshaped-eq}
\left \langle \nabla u, Z \right\rangle < - M < 0
\end{equation}
at any point where $\nabla u$ exists.
\end{theorem}

\begin{proof}
As already declared, we prove the result for $p \neq Q$, addressing the reader to Remark \ref{p=Q} for indications about the straightforward extension to the case $p = Q$.
Consider, for any $z \in \partial \Omega_1$ the gauge ball $B_\rho(z_1, R)$ contained in $\Heis^n \setminus\overline{\Omega_1}$ and touching $\partial\Omega_1$  in $z$. Similarly, for $z \in \partial \Omega_2$, consider $B_\rho(z_2, R)$ contained in ${\Omega_2}$ and touching $\partial\Omega_2$ in $z$. These tangent gauge balls, with uniform radius $R$, exist by assumption, see Definition \ref{gauge-ball-def}. 
On $B_\rho(z_1, R)$, define a function $v_1$ satisfying $v =1 $ on $B_\rho (z_1, R/2)$ and 
\begin{equation}
\label{barrier-function}
v_1 (\cdot) = \alpha \, \rho(z_1^{-1} \cdot)^{-\frac{Q - p}{p-1}} + \beta,
\end{equation}
on $B_\rho(z_1, R) \setminus \overline B_\rho(z_1, R/2)$,
where the constants $\alpha$ and $\beta$ are chosen so that $v_1 = 0$ on $\partial B_\rho(z_1, R)$ and $v_1 = 1$ on $\partial B_\rho(z_1, R/2)$. Analogously, define on $B_\rho(z_2, R)$ a function $v_2$ satisfying $v_2 = 1$ on $B_\rho (z_2, R/2)$ and
\begin{equation}
v_2 (\cdot) = \alpha \, \rho(z_2^{-1} \cdot)^{-\frac{Q - p}{p-1}} + \beta,
\end{equation}
on $B_\rho(z_2, R) \setminus \overline B_\rho(z_2, R/2)$,
where the constants $\alpha$ and $\beta$ are chosen so that $v_2 = 0$ on $\partial B_\rho(z_2, R)$ and $v_2 = 1$ on $\partial B_\rho(z_2, R/2)$. Explicitly, we have
\begin{equation}
\alpha = \frac{R^{\frac{Q - p}{p-1}}}{2^{\frac{Q - p}{p-1}} - 1} , \qquad \beta = - \frac{1}{2^{\frac{Q - p}{p-1}} - 1}.
\end{equation}

Observe now that the function $v_1$ and $v_2$ are smooth up to the boundary in $\overline{B_\rho (z_1, R)} \setminus B_\rho(z_1, R/2)$ and $\overline{B_\rho (z_2, R)} \setminus B_\rho(z_2, R/2)$ respectively, and they both enjoy nonvanishing gradient in these sets. Actually, a direct computation shows that
\begin{equation}
\label{lowerbound-gradient}
\abs{\nabla v_1}(w_1) \geq C , \qquad \abs{\nabla v_2}(w_2) \geq C  
\end{equation}
for any $w_1 \in \overline{B_\rho (z_1, R)} \setminus B_\rho(z_1, R/2)$ and any $w_2 \in \overline{B_\rho (z_2, R)} \setminus B_\rho(z_2, R/2)$, where the constant $C$ does not depend on $w_1$ nor on $w_2$.
Such gradients being nonvanishing, combined with  $\partial B_\rho (z_1, R)$ and $\partial B_\rho (z_2, R)$  being regular level sets of $v_1$ and $v_2$, imply, on the one hand, that
\begin{equation}
\label{limit-barrier1}
\lim_{B_\rho(z, R) \ni w \to z}\frac{\nabla v_1}{\abs{\nabla v_1}}(w) =  \nu_{\Omega_1} (z),
\end{equation}
where $z \in \partial \Omega_1$ and $\nu_{\Omega_1} (z)$ is the Euclidean outward unit normal to $\Omega_1$, and on the other hand that
 \begin{equation}
\label{limit-barrier2}
\lim_{B_\rho(z, R) \ni w \to z}\frac{\nabla v_2}{\abs{\nabla v_2}}(w) =  -\nu_{\Omega_2} (z),
\end{equation}
where this time $z \in \partial \Omega_2$ and $\nu_{\Omega_2} (z)$ is the Euclidean outward unit normal to $\Omega_2$. In getting \eqref{limit-barrier1} and \eqref{limit-barrier2}, we again used the tangency property of the gauge balls with respect to the boundaries of $\Omega_1$ and $\Omega_2$. Observe now that there exists $K > 0$ such that 
\begin{equation}
\label{strict-star-dim}
\left\langle \nu, Z\right\rangle (z) \geq K
\end{equation}
 for any $z \in \partial \Omega_1 \cup \partial \Omega_2$,  as follows from the strict starshapedness of $\Omega_1$ and $\Omega_2$, their boundedness and the $C^1$-regularity of their boundaries.
Now, Proposition \ref{entrata-palla}, the limits \eqref{limit-barrier1} and \eqref{limit-barrier2}, the uniform lower bounds on the gradients of $v_1$ and $v_2$ \eqref{lowerbound-gradient}, and  \eqref{strict-star-dim},
imply that
\begin{equation}
\label{uniform-barrier1}
\frac{v_1 (\delta_\lambda (z)) - v_1(z)}{\lambda} = \frac{v_1(\delta_\lambda (z))}{\lambda} \geq \frac{1}{2} C K,  
\end{equation}
for any $z \in \partial \Omega_1$
and analogously
\begin{equation}
\label{uniform-barrier2}
\frac{v_2(\delta_{-\lambda} (z))}{\lambda} \geq \frac{1}{2} C K
\end{equation}
for any $z \in \Omega_2$, for any $0 < \lambda < \overline{\lambda}$. Importantly, observe that $\overline{\lambda}$ can be made independent of $z \in \partial \Omega_1 \cup \partial \Omega_2$, as it immediately follows from the continuity properties of such parameter stated in Proposition \ref{entrata-palla} and the compactness of $\partial \Omega_1 \cup \partial \Omega_2$.
\smallskip

As observed in Corollary \ref{u-strict}, $0 < u < 1$ in the open annulus, and thus, combining this information with the continuity of such function, we deduce that there exists a constant $0 < L < 1$ such that
\begin{equation}
\label{strong-applied}
\frac{1}{L} \leq u (w) \leq 1 - L
\end{equation}
for any $w \in \partial B_\rho(z, R/2)$. A straightforward compactness argument involving the continuity of $u$ shows that $L$ can be chosen independently of $z$.
Consider now, for $z \in \partial \Omega_1$, the $p$-harmonic functions $u$ and $1 - L v_1$  on $B_\rho(z, R) \setminus \overline{B_\rho(z, R/2)}$. Observe that $u \leq  1 =1 - Lv_1$ on $\partial B_\rho (z, R)$, since $u \leq 1$ on the whole annulus $\overline{\Omega_2} \setminus \Omega_1$ and $v_1 = 0$ on $\partial B_\rho (z, R)$ by construction. Moreover, $u \leq 1 - L = 1 - L v_1$ on $\partial B_\rho (z, R/2)$ in light of the second inequality in \eqref{strong-applied} and again by construction of $v_1$. Then, the comparison principle for $p$-harmonic functions recalled in Proposition \ref{comparison} combined with \eqref{uniform-barrier1} implies
\begin{equation}
\label{barrier-applied1}
\frac{1 - u(\delta_\lambda(z))}{\lambda} \geq  \frac{L v_1(\delta_\lambda(z))}{\lambda} \geq  \frac{1}{2} LCK
\end{equation}  
for any $0 < \delta < \overline{\lambda}$.
 
Arguing very similarly in comparing the functions $u$ and $Lv_2$ in the annulus $B_\rho(z, R) \setminus \overline{B_\rho(z, R/2)}$ with $z \in \partial \Omega_2$, we get, using the first inequality in \eqref{strong-applied} and the definition of $v_2$, that
\begin{equation}
\label{barrier-applied2}
\frac{u(\delta_{-\lambda}(z))}{\lambda} \geq \frac{L v_2(\delta_{-\lambda} (z))}{\lambda} \geq \frac{1}{2} L C K,
\end{equation}
again for any $0 < \lambda < \overline{\lambda}$.

Consider then the function $u_\lambda (w) = u(\delta_\lambda (w))$ on $\Omega_2^{-\lambda} \setminus \overline{\Omega}_1$. Recall that by $\Omega_2^\lambda$ we denote the contraction of $\Omega_2$ through dilations, as defined in \eqref{dilation-subset}. By Lemma \ref{dilation-invariance}, $u_\lambda$ is $p$-harmonic and observe that on $\partial \Omega_1$ we have 
\[
\frac{u_\lambda}{\lambda} \leq \frac{1}{\lambda} - \frac{1}{2}L C K = \frac{u}{\lambda} - \frac{1}{2}L C K
\]
by \eqref{barrier-applied1}, and on $\partial \Omega_2^{-\lambda}$ we have
\[
0 = \frac{u_\lambda}{\lambda} \leq \frac{u}{\lambda} - \frac{1}{2} L C K,
\]
by \eqref{barrier-applied2}.
Thus, applying the Comparison Principle to the $p$-harmonic functions $u_\lambda/\lambda$ and $u/\lambda - LCK/2$ we get, for any $w \in \Omega_2 \setminus\overline{\Omega}_1$, that
\begin{equation}
\label{quasi-fine}
\frac{u(\delta_\lambda (w)) - u(w)}{\lambda} \leq  - \frac{1}{2}{L C K}
\end{equation}
for any $0 < \lambda < \overline{\lambda}$.
Assume now that $\nabla u$ exists at $w$. Then, we have
\begin{equation}
\label{limit-lambda}
\lim_{\lambda \to 0^+} \frac{u(\delta_\lambda (w)) - u(w)}{\lambda} = \left \langle \nabla u (w), {\frac{d}{d \lambda}\delta_\lambda (w)_{\big\vert \lambda = 0} } \right\rangle = \left\langle{\nabla u, Z}\right\rangle(w),
\end{equation}
where in the last step we used \eqref{flow-z}.
We thus conclude, coupling \eqref{quasi-fine} with \eqref{limit-lambda}, that
\begin{equation}
\label{superquasi-fine}
\left\langle{\nabla u, Z}\right\rangle(w) < - \frac{1}{2}{L C K} < 0
\end{equation}
at $w \in \Omega_2 \setminus \overline{\Omega_1}$.
Observe that the upper bound in \eqref{superquasi-fine} does not depend on the particular point $w$ where  $\nabla u$ exists, and thus it completes the proof of Theorem \ref{thm-prestarshaped}.
\end{proof}
Let us finally briefly show how Theorem \ref{th-starshaped-bounded} follows as a corollary.

\begin{proof}[Proof of Theorem \ref{th-starshaped-bounded}]
The uniform negative upper bound for $\left \langle \nabla u, Z \right \rangle$ holding true at any point of standard diffentiability for $u$ implies that if in addition such function is $C^1$, we have 
\begin{equation}
\label{uniform-starshaped}
\left \langle \nabla u, Z \right \rangle < 0
\end{equation}
on the whole $\Omega_2 \setminus \overline{\Omega}_1$.
In particular, $\nabla u$ never vanishes in the open annulus, and thus the sets $\{u \geq t\} \cup \overline{\Omega}_1$ for $t \in (0, 1)$ are bounded by the $C^1$ submanifolds $\{u = t\}$ with exterior pointing unit normal at any $w_t \in \{u = t\}$  given by $\nu_t = -\nabla u/ \abs{\nabla u}$ computed at such point. This information, plugged in \eqref{uniform-starshaped}, yields
\begin{equation}
\label{fine}
\left\langle \nu_t, Z\right\rangle > 0,
\end{equation}
that is, according to Definition \ref{starshaped}, the sets $\{u \geq t\} \cup \overline{\Omega}_1$ are strictly starshaped for any $t \in (0, 1)$.  

\end{proof}

\medskip

\subsection*{Acknowledgements}\emph{The authors are grateful to C.~Bianchini, G.~Ciraolo, F.~Dragoni and D.~Ricciotti for useful discussions during the preparation of the manuscript. The authors would like to thank the reviewer for his/her detailed comments that helped us to improve the manuscript.
The authors are members of Gruppo Nazionale per l'Analisi Matematica, 
la Probabilit\`a e le loro Applicazioni (GNAMPA), which is part of the 
Istituto Nazionale di Alta Matematica (INdAM),
and they are partially funded by the GNAMPA project 
``Aspetti geometrici in teoria del potenziale 
lineare e nonlineare''.}



\end{document}